\documentclass[]{article}
\usepackage{graphicx}
\usepackage{amsfonts}
\usepackage{amsmath}
\usepackage{amssymb}
\usepackage{fancyhdr}
\usepackage{titlesec}
\usepackage{indentfirst}
\usepackage{booktabs}
\usepackage{verbatim}
\usepackage{color}
\usepackage{amsthm}

\usepackage{pdflscape}

\usepackage{todonotes}

\usepackage[page,header]{appendix}
\usepackage{titletoc}

\newcommand{\rd}{\,\mathrm{d}}

\def\mQ{\mathcal{Q}}
\def\mT{\mathcal{T}}

\def\mS{\mathcal{S}}
\def\mG{\mathcal{G}}
\def\mM{\mathcal{M}}

\numberwithin{equation}{section}
\newtheorem{theorem}{Theorem}[section]

\newtheorem{proposition}[theorem]{Proposition}

\newtheorem{remark}[theorem]{Remark}

\begin{document}

	\title{A second-order asymptotic-preserving and positivity-preserving exponential Runge-Kutta method for a class of stiff kinetic equations\footnote{This research was supported by NSF grant DMS-1620250 and NSF CAREER grant DMS-1654152. Support from DMS-1107291: RNMS KI-Net is also gratefully acknowledged.}}
	
	\author{Jingwei Hu\footnote{Department of Mathematics, Purdue University, West Lafayette, IN 47907, USA (jingweihu@purdue.edu).} \  \  
	    and \ Ruiwen Shu\footnote{Department of Mathematics, University of Maryland, College Park, MD 20742, USA (rshu@cscamm.umd.edu).}}    
	\maketitle

\begin{abstract}
We introduce a second-order time discretization method for stiff kinetic equations. The method is asymptotic-preserving (AP) -- can capture the Euler limit without numerically resolving the small Knudsen number; and positivity-preserving -- can preserve the non-negativity of the solution which is a probability density function for arbitrary Knudsen numbers. The method is based on a new formulation of the exponential Runge-Kutta method and can be applied to a large class of stiff kinetic equations including the BGK equation (relaxation type), the Fokker-Planck equation (diffusion type), and even the full Boltzmann equation (nonlinear integral type). Furthermore, we show that when coupled with suitable spatial discretizations the fully discrete scheme satisfies an entropy-decay property. Various numerical results are provided to demonstrate the theoretical properties of the method.
\end{abstract}

{\small 
{\bf Key words.}  stiff kinetic equation, exponential Runge-Kutta method, asymptotic-preserving, positivity-preserving, entropy-decay

{\bf AMS subject classifications.} 82C40, 65L04, 35Q31, 35Q84, 35Q20, 65L06, 65F60
}

%\tableofcontents	

%%%%%%%%%%%%%%%%%%%%%%%%%%%%%%%%%%%%%%	
\section{Introduction}
\label{sec:intro}
	
Kinetic equations describe the non-equilibrium dynamics of a gas or system comprised of a large number of particles. In multiscale modeling hierarchy, they serve as a bridge that connects microscopic Newtonian mechanics and macroscopic continuum mechanics. In this paper, we are concerned with the following class of kinetic equations:
\begin{equation} \label{KE}
\partial_t f+v \cdot \nabla_x f= \frac{1}{\varepsilon}\mathcal{Q}(f), \quad t\geq 0, \quad x\in \Omega \subset \mathbb{R}^{d}, \quad v\in \mathbb{R}^{d},
\end{equation}	
where $f=f(t,x,v)$ is the one-particle probability density function (PDF) of time $t$, position $x$, and particle velocity $v$. $\mathcal{Q}$ is the collision operator which acts only in the velocity space and models the interactions between particles. Examples of $\mathcal{Q}$ include: the Boltzmann collision operator (a nonlinear integral operator) \cite{Cercignani}, the BGK operator (a relaxation type operator) \cite{BGK54}, the kinetic Fokker-Planck operator (a diffusion type operator) \cite{Villani02}, among others. Finally, $\varepsilon$ is the Knudsen number defined as the ratio of the mean free path and typical length scale. The magnitude of $\varepsilon$ indicates the degree of rarefaction of the system. When $\varepsilon$ is small, collisions happen very frequently so that the system is close to the fluid regime. In fact, one can derive the compressible Euler equations from (\ref{KE}) as the leading-order asymptotics by sending $\varepsilon \rightarrow 0$.

When $\varepsilon$ is small, numerically solving the equation (\ref{KE}) is challenging due to the stiff collision term on the right hand side. Any explicit time discretization would suffer from severe stability constraint (time step $\Delta t$ has to be $O(\varepsilon)$). As such, schemes that can remove this constraint are highly desirable. The so-called asymptotic-preserving (AP) scheme \cite{Jin99} is exactly designed for this kind of problems: it solves the kinetic equation without resolving small scales ($\Delta t$ can be chosen independent of $\varepsilon$), yet when $\varepsilon \rightarrow 0$ while keeping $\Delta t$ fixed, it automatically becomes a macroscopic fluid solver, i.e., a consistent discretization to the limiting Euler equations (see \cite{Jin_Rev, HJL17} for a comprehensive review of AP schemes).

The AP property is certainly a desired feature for handling multiscale kinetic equations, especially in the near fluid regime. However, most of AP schemes require some implicit treatment or reformulation of the equation such that the positivity of the solution is lost during the construction. This is unphysical since $f$ is a PDF, and sometimes even causes the simulation to break down. The design of high order (at least second order) schemes that are both AP and positivity-preserving turns out to be highly nontrivial and needs to be handled in a problem-dependent basis. Recently, we developed a family of second-order AP and positivity-preserving schemes for the stiff BGK equation \cite{HSZ18}. The method is based on the implicit-explicit (IMEX) Runge-Kutta framework plus a key correction step utilizing the special structure of the BGK operator. It also works for some hyperbolic systems but is limited to relaxation type operators. 

In this paper, we propose a more general time discretization method based on a new exponential Runge-Kutta formulation that can be applied to a large class of stiff kinetic equations including the BGK, the Fokker-Planck, and even the full Boltzmann equations. To summarize, our method possesses the following features:
\begin{itemize}
\item The scheme is second-order accurate in the kinetic regime $\varepsilon=O(1)$;
\item The scheme is AP: for fixed $\Delta t$, when $\varepsilon \rightarrow 0$, it reduces to a second-order scheme for the limiting Euler equations (in fact, the limiting scheme can be made as the optimal second-order strong-stability-preserving (SSP) Runge-Kutta method, i.e., the improved Euler or Heun's method \cite{GST01, GKS11});
\item The scheme is positivity-preserving for any $\varepsilon\geq 0$: if $f^n\geq 0$, then $f^{n+1}\geq 0$;
\item The time step of the scheme is only constrained by the transport part and can be chosen the same as in the forward Euler method;
\item The scheme satisfies an entropy-decay property when coupled with suitable spatial discretizations.
\end{itemize}

The rest of this paper is organized as follows. In Section~\ref{sec:new}, we construct the method for the general kinetic equation (\ref{KE}) without specifying the collision operator. The emphasis is to make the method second order and positivity-preserving. In Section~\ref{sec:AP}, we consider the application of the method to specific kinetic equations and discuss its AP property. A comparison with existing similar methods is given as well. In Section~\ref{sec:homo}, we address the issue of solving the homogeneous equation ((\ref{KE}) without transport term) which is an important building block of the proposed method. In Section~\ref{sec:entropy}, we prove the entropy-decay property of the method when coupled with suitable spatial discretizations. Some remarks regarding the spatial and velocity domain discretizations are given in Section~\ref{sec:spatial}. Numerical examples are presented in Section~\ref{sec:num}. The paper is concluded in Section~\ref{sec:con}.
	
%%%%%%%%%%%%%%%%%%%%%%%%%%%%%%%%%%%%%%		
	
\section{A new exponential Runge-Kutta method for general stiff kinetic equations}
\label{sec:new}

We now present the procedure to construct the new exponential Runge-Kutta method. Since the method is quite general and can be applied to a large class of kinetic equations, we will start with the equation (\ref{KE}) without specifying the collision operator and derive a scheme that is both second-order accurate and positivity-preserving. Then in Section~\ref{sec:AP}, we will consider specific collision operators and discuss the AP property of the scheme as this latter part is problem dependent.

To begin with, let us introduce the following notation: for the autonomous ODE
\begin{equation}
\frac{\rd{}}{\rd{t}}f = A(f),\quad f|_{t=t_0} = g,
\end{equation}
where $A$ is an operator, either linear or nonlinear, we use $\exp(sA)g, \ s\ge 0$ to represent its solution at $t=t_0+s$.\footnote{Note that $\exp(sA)$ is merely a symbol to denote the map from the solution at time $t_0$ to the solution at time $t_0+s$, and should not be understood as the matrix exponential except the linear case.}

We now consider an ODE resulting from the semi-discretization of the equation (\ref{KE}) (only space $x$ is discretized while time $t$ and velocity $v$ are left continuous):
\begin{equation}\label{ode}
\frac{\rd{}}{\rd{t}}f = \mT (f)+\frac{1}{\varepsilon}\mQ (f).
\end{equation}
Here $\mT(f)$ is a discretized operator for the transport term $-v\cdot\nabla_x f$ and $\mQ(f)$ is the collision operator which may take various forms depending on the application. We assume the operators $\mT(f)$ and $\mQ(f)$ are positivity-preserving. To be precise,
\begin{itemize}
\item for $\mT(f)$, we assume
\begin{equation} \label{cflcond}
f \ge 0\, \Longrightarrow \, f+a \Delta t\,\mT (f) \ge 0,\quad \forall \ \text{constant} \ a \ \ \text{s.t.} \ 0\leq a\Delta t\leq \Delta t_{\text{FE}},
\end{equation} 
where $\Delta t_{\text{FE}}$ is the maximum time step allowance such that the forward Euler method is positivity-preserving;
\item for $\mQ(f)$, we assume the solution to the homogeneous equation
\begin{equation}\label{homo_ode}
\frac{\rd{}}{\rd{t}}f = \mQ (f)
\end{equation}
satisfies $f\ge 0$ for all $t\ge t_0$, if the initial data $f|_{t=t_0}=g\ge 0$. In other words, 
\begin{equation} \label{ppcond}
g\ge 0 \, \Longrightarrow \, \exp(s\mQ)g\ge 0, \quad \forall \ \text{constant} \  s\ge 0.
\end{equation}
\end{itemize}

\begin{remark}
The condition (\ref{cflcond}) can be easily satisfied if a positivity-preserving spatial discretization is used, as was done in \cite{HSZ18}. The condition (\ref{ppcond}) is a theoretical property that holds for any kinetic equations.
\end{remark}

We are ready to construct the numerical method for equation (\ref{ode}). For the time being, we assume that the solution to the homogeneous equation (\ref{homo_ode}) can be found exactly and will get back to this in Section~\ref{sec:homo} when discussing specific models. We propose an exponential Runge-Kutta scheme of the following form:
\begin{equation}\label{scheme}
\begin{split}
& f^{(0)} = \exp\left(a_0 \Delta t \frac{1}{\varepsilon}\mQ\right)f^n, \\
& f^{(1)} = \exp\left(a_1 \Delta t \frac{1}{\varepsilon}\mQ\right)\left(f^{(0)} + b_1 \Delta t \mT (f^{(0)})\right), \\
& f^{(2)} = f^{(1)} + b_2 \Delta t \mT (f^{(1)}), \\
& f^{n+1} = \exp\left(a_2 \Delta t \frac{1}{\varepsilon}\mQ\right)\left[ w f^{(2)} + (1-w) \exp\left((1-a_2)\Delta t \frac{1}{\varepsilon}\mQ\right)f^n\right],
\end{split}
\end{equation}
where the constants $a_0$, $a_1$, $a_2$, $b_1$, $b_2$, and $w$ are to be determined. 

With the previous assumptions on $\mT$ and $\mQ$, it is easy to see 
\begin{proposition} \label{prop_pos}
The scheme (\ref{scheme}) is positivity-preserving, i.e., if $f^n\geq 0$, then $f^{n+1}\geq 0$ provided
\begin{equation} \label{pos_cond}
a_0,a_1,b_1,b_2\ge0, \quad 0\le a_2,w\le 1,
\end{equation}
under the CFL condition 
\begin{equation} \label{CFL}
\Delta t \leq \frac{\Delta t_{\text{FE}}}{\max(b_1,b_2)},
\end{equation}
and the ratio is understood as infinite if the denominator is zero.
\end{proposition}

We next derive the conditions for (\ref{scheme}) to be second order in the kinetic regime. Without loss of generality, we assume $\varepsilon=1$.

First of all, given the solution $f^n=f(t_n)$, if we Taylor expand the exact solution of (\ref{ode}) at $t_{n+1}$ around $t_n$, we have
\begin{equation}\label{exact}
\begin{split}
\exp(\Delta t (\mT+\mQ))f^n = & f^n + \Delta t (\mT(f^n)+\mQ(f^n)) \\ & + \frac{1}{2}\Delta t^2 (\mT'(f^n)\mT(f^n) + \mT'(f^n)\mQ(f^n) + \mQ'(f^n)\mT(f^n) + \mQ'(f^n)\mQ(f^n)) \\&+ O(\Delta t^3),
\end{split}
\end{equation}
where $\mQ'$, $\mT'$ are the Fr\'{e}chet derivative of $\mQ$ and $\mT$.

Similarly the exact solution of (\ref{homo_ode}) at $t_{n+1}$ is approximated by 
\begin{equation} 
\exp(\Delta t \mQ)f^n = f^n + \Delta t \mQ(f^n) + \frac{1}{2}\Delta t^2 \mQ'(f^n)\mQ(f^n) + O(\Delta t^3).
\end{equation} 
Using this in the first equation of (\ref{scheme}), we have
\begin{equation}
\begin{split}
f^{(0)} = & f^n +a_0 \Delta t \mQ(f^n) +  \frac{1}{2}a_0^2\Delta t^2 \mQ'(f^n)\mQ(f^n) + O(\Delta t^3).
\end{split}
\end{equation}
Continuing the Taylor expansion of $f^{(1)}$, $f^{(2)}$, and $f^{n+1}$ in (\ref{scheme}), we have
\begin{equation}
\begin{split}
f^{(1)} = & (f^{(0)} + b_1 \Delta t \mT (f^{(0)})) +a_1\Delta t \mQ(f^{(0)} + b_1 \Delta t \mT (f^{(0)})) \\ & +\frac{1}{2}a_1^2\Delta t^2 \mQ'(f^{(0)} + b_1 \Delta t \mT (f^{(0)}))\mQ(f^{(0)} + b_1 \Delta t \mT (f^{(0)})) + O(\Delta t^3) \\ 
= & f^n + \Delta t( (a_0+a_1)\mQ(f^n) + b_1\mT (f^n)) \\ & + \Delta t^2\left(b_1a_0\mT'(f^n)\mQ(f^n) +  a_1b_1\mQ'(f^n)\mT(f^n) + \frac{1}{2}(a_0+a_1)^2 \mQ'(f^n)\mQ(f^n)\right) + O(\Delta t^3).
\end{split}
\end{equation}
\begin{equation}
\begin{split}
f^{(2)} = & f^{(1)} + b_2 \Delta t \mT (f^{(1)}) \\=& f^n + \Delta t( (a_0+a_1)\mQ(f^n) + b_1\mT (f^n)) \\ & + \Delta t^2\left(b_1a_0\mT'(f^n)\mQ(f^n) +  a_1b_1\mQ'(f^n)\mT(f^n) + \frac{1}{2}(a_0+a_1)^2 \mQ'(f^n)\mQ(f^n)\right)  \\ & + b_2\Delta t\mT(f^n) + b_2\Delta t^2\mT'(f^n)((a_0+a_1)\mQ(f^n) + b_1\mT(f^n)) + O(\Delta t^3) \\
= & f^n + \Delta t( (a_0+a_1)\mQ(f^n) + (b_1+b_2)\mT (f^n)) \\ & + \Delta t^2\left(b_1b_2\mT'(f^n)\mT(f^n) + (b_1a_0+b_2a_0+b_2a_1)\mT'(f^n)\mQ(f^n)\right. \\ & \left.+  a_1b_1\mQ'(f^n)\mT(f^n) + \frac{1}{2}(a_0+a_1)^2 \mQ'(f^n)\mQ(f^n)\right)  + O(\Delta t^3).
\end{split}
\end{equation}
\begin{equation}
\begin{split}
wf^{(2)} + & (1-w) \exp((1-a_2)\Delta t \mQ)f^n = w\left[f^n + \Delta t( (a_0+a_1)\mQ(f^n) + (b_1+b_2)\mT (f^n)) \right.\\ & + \Delta t^2\left(b_1b_2\mT'(f^n)\mT(f^n) + (b_2a_1+b_1a_0+b_2a_0)\mT'(f^n)\mQ(f^n)\right. \\ &\left.\left. + a_1b_1\mQ'(f^n)\mT(f^n) + \frac{1}{2}(a_0+a_1)^2 \mQ'(f^n)\mQ(f^n)\right)\right] \\ & + (1-w)\left[f^n+(1-a_2)\Delta t \mQ(f^n)+\frac{1}{2}(1-a_2)^2\Delta t^2 \mQ'(f^n)\mQ(f^n)\right] + O(\Delta t^3) \\
= & f^n + \Delta t[ (w(a_0+a_1)+(1-w)(1-a_2))\mQ(f^n) + w(b_1+b_2)\mT (f^n)] + \Delta t^2\left[wb_1b_2\mT'(f^n)\mT(f^n) \right. \\ & +w(b_2a_1+b_1a_0+b_2a_0)\mT'(f^n)\mQ(f^n) + wa_1b_1\mQ'(f^n)\mT(f^n) \\ & \left.+ \frac{1}{2}(w(a_0+a_1)^2+(1-w)(1-a_2)^2) \mQ'(f^n)\mQ(f^n)\right] + O(\Delta t^3). 
\end{split}
\end{equation}
Finally,
\begin{equation}\label{accu}
\begin{split}
f^{n+1} = &  f^n + \Delta t[ (w(a_0+a_1)+(1-w)(1-a_2))\mQ(f^n) + w(b_1+b_2)\mT (f^n)] + \Delta t^2\left[wb_1b_2\mT'(f^n)\mT(f^n)\right.  \\ & +w(b_2a_1+b_1a_0+b_2a_0)\mT'(f^n)\mQ(f^n) + wa_1b_1\mQ'(f^n)\mT(f^n)\\ &\left.+ \frac{1}{2}(w(a_0+a_1)^2+(1-w)(1-a_2)^2) \mQ'(f^n)\mQ(f^n)\right]  \\ & + a_2\Delta t [\mQ(f^n) + \Delta t \mQ'(f^n)( (w(a_0+a_1)+(1-w)(1-a_2))\mQ(f^n) + w(b_1+b_2)\mT (f^n)) ] \\ & + \frac{1}{2}a_2^2\Delta t^2 \mQ'(f^n)\mQ(f^n) + O(\Delta t^3) \\
= &  f^n + \Delta t[ (w(a_0+a_1+a_2)+(1-w))\mQ(f^n) + w(b_1+b_2)\mT (f^n)] + \Delta t^2\left[wb_1b_2\mT'(f^n)\mT(f^n) \right. \\ & +w(b_2a_1+b_1a_0+b_2a_0)\mT'(f^n)\mQ(f^n) + w(a_1b_1+a_2b_2+a_2b_1)\mQ'(f^n)\mT(f^n) \\ & \left.+ \frac{1}{2}(w(a_0+a_1+a_2)^2 + (1-w)) \mQ'(f^n)\mQ(f^n)\right] + O(\Delta t^3). \\ 
\end{split}
\end{equation}

Comparing (\ref{exact}) and (\ref{accu}), we arrive at the following order conditions:
\begin{equation}
\begin{split}
&w(a_0+a_1+a_2)+(1-w)=1; \quad w(b_1+b_2)=1; \quad wb_1b_2=\frac{1}{2};\\
&w(b_2a_1+b_1a_0+b_2a_0)=\frac{1}{2}; \quad w(a_1b_1+a_2b_2+a_2b_1)=\frac{1}{2};\\
&w(a_0+a_1+a_2)^2+(1-w)=1.
\end{split}
\end{equation}
Further simplification yields
\begin{proposition} \label{prop_ord}
The scheme (\ref{scheme}) is second-order accurate for $\varepsilon=O(1)$ provided
\begin{align} 
& a_0+a_1+a_2 = 1, \label{ord_cond1}\\
& w(b_1+b_2) = 1, \label{ord_cond2}\\
& wb_1b_2 = \frac{1}{2}, \label{ord_cond3}\\
& w(b_2a_1+(b_1+b_2)a_0) = \frac{1}{2}.\label{ord_cond4}
\end{align}
\end{proposition}

Combining the positivity conditions and order conditions found in Propositions \ref{prop_pos} and \ref{prop_ord}, one can obtain a second-order positivity-preserving scheme for equation (\ref{ode}). To find a set of parameters satisfying these conditions, first notice that (\ref{ord_cond2}) and (\ref{ord_cond3}) imply $b_1,b_2$ are the solutions of the quadratic equation
\begin{equation}
b^2-\frac{1}{w}b + \frac{1}{2w} = 0,
\end{equation}
whose solutions are given by
\begin{equation}\label{b12}
b_{1,2} = \frac{1}{1\pm\sqrt{1-2w}}, \quad \text{for} \ 0<w\leq \frac{1}{2}.
\end{equation}
In order to obtain the best CFL condition (minimize $\max(b_1,b_2)$ in (\ref{CFL})), we choose 
\begin{equation} \label{coef}
w=\frac{1}{2},\quad b_1=b_2=1,
\end{equation}
hence the CFL condition (\ref{CFL}) is the same as the forward Euler method. Then (\ref{ord_cond1}) and (\ref{ord_cond4}) reduce to
\begin{equation} \label{coef1}
a_0+a_1+a_2=1,\quad a_0=a_2.
\end{equation}
To insure positivity, we only need additionally $a_0,a_1\geq 0$, $0\leq a_2 \leq 1$ (see (\ref{pos_cond})). However, to obtain a good AP property, we require
\begin{equation} \label{coef2}
a_0,a_1>0, \quad 0< a_2<1.
\end{equation}
This will be further elaborated in Section~\ref{sec:AP}.
One choice of $a_0,a_1,a_2$ is
\begin{equation} \label{coef3}
a_0=a_1=a_2=\frac{1}{3}.
\end{equation}

\begin{remark}
If one sets $\mQ=0$, then (\ref{scheme}) becomes an explicit second-order SSP Runge-Kutta scheme applied to the purely transport equation; moreover, our choice (\ref{coef}) just gives the standard optimal one, i.e., the improved Euler or Heun's method \cite{GST01, GKS11}. In what follows, we will refer this scheme as SSP-RK2. If one sets $\mT=0$, then (\ref{scheme}) becomes $f^{n+1}=\exp\left(\Delta t \frac{1}{\varepsilon} \mQ\right)f^n$ which is the exact solution to the homogeneous equation (\ref{homo_ode}).
\end{remark}

\begin{remark}
For $a_0,a_1>0$ and $0<a_2<1$, (\ref{scheme}) would require 4 times evaluation of the operator $\exp(s\mQ)$ in each time step. However, similar to the Strang splitting, one can combine the operator $\exp\left(a_2\Delta t\frac{1}{\varepsilon}\mQ\right)$ in the last stage of the $n$-th step with the operator $\exp\left(a_0\Delta t\frac{1}{\varepsilon}\mQ\right)$ in the first stage of the $(n+1)$-th step, so that effectively one only needs 3 times of such evaluations in each time step.
\end{remark}

%%%%%%%%%%%%%%%%%%%%%%%%%%%%%%%%%%%%%%%%

\section{Application to specific kinetic equations and AP property}
\label{sec:AP}

By now, we have obtained a second-order positivity-preserving scheme ((\ref{scheme}) with coefficients satisfying (\ref{coef}) (\ref{coef1}) (\ref{coef2})) for the general stiff kinetic equation (\ref{ode}). In this section, we apply the scheme to some specific kinetic equations and discuss its AP property.

We will consider the equation (\ref{ode}) with the following collision operators:
\begin{itemize}
\item The BGK operator \cite{BGK54}, a simple relaxation type operator used to mimic the complicated Boltzmann collision operator:
\begin{equation}\label{BGK}
\mQ(f) = \eta(\mM[f] - f),
\end{equation}
where $\mM[f]$ is the Maxwellian defined by
\begin{equation} \label{Max}
\mM[f]=\frac{\rho}{(2\pi T)^{\frac{d}{2}}}\exp \left( -\frac{|v-u|^2}{2T}\right),
\end{equation}
with the density $\rho$, bulk velocity $u$, and temperature $T$ given by the moments of $f$:
\begin{equation} \label{mom}
\rho=\int_{\mathbb{R}^d}f\,\rd{v}, \quad u=\frac{1}{\rho}\int_{\mathbb{R}^d}f\,\rd{v}, \quad T=\frac{1}{d\rho}\int_{\mathbb{R}^d}f|v-u|^2\,\rd{v},
\end{equation}
and $\eta$ is some positive function depending only on $\rho$ and $T$.

\item The ES-BGK operator \cite{Holway66}, a generalized BGK model used to fit realistic values of the transport coefficients:
\begin{equation}\label{ESBGK}
\mQ(f) = \eta(\mG[f]-f),
\end{equation}
where $\mG[f]$ is a Gaussian function defined by
\begin{equation}
\mG[f] = \frac{\rho}{\sqrt{\det(2\pi \bar{T})}}\exp\left(-\frac{1}{2}(v-u)^T \bar{T}^{-1}(v-u)\right),
\end{equation}
with $\rho$, $u$, and $T$ given in (\ref{mom}) and 
\begin{equation}\label{barT}
\bar{T} = (1-\nu)TI+\nu\Theta,\quad \Theta=\frac{1}{\rho}\int_{\mathbb{R}^d} f(v-u)\otimes (v-u)\,\rd{v},
\end{equation}
where $-\frac{1}{2}\leq \nu <1$ is a parameter and $I$ is the identity matrix. $\eta$ is some positive function of $\rho$ and $T$.

\item The Boltzmann collision operator \cite{Cercignani}, a fundamental equation in kinetic theory describing the binary collisions in a rarefied gas:
\begin{equation} \label{Boltz}
\mQ(f) = \int_{\mathbb{R}^d}\int_{S^{d-1}}B(v-v_*, \sigma) [f(v')f(v_*')-f(v)f(v_*)]\,\rd{\sigma}\,\rd{v_*},
\end{equation}
where $v'$ and $v_*'$ (post-collisional velocities) are defined in terms of $v$ and $v_*$ (pre-collisional velocities) as
\begin{equation}
v'=\frac{v+v_*}{2}+\frac{|v-v_*|}{2}\sigma, \quad v_*'=\frac{v+v_*}{2}-\frac{|v-v_*|}{2}\sigma,
\end{equation}
with $\sigma$ being a vector varying on the unit sphere $S^{d-1}$. $B$ is the collision kernel characterizing the scattering rate and is a non-negative function.

\item The kinetic Fokker-Planck operator \cite{Villani02}, a kinetic model describing the drift and diffusion effects of particles: 
\begin{equation}\label{FP}
\mQ(f) = \nabla_v\cdot \left(\mM[f]\nabla_v \frac{f}{\mM[f]}\right),
\end{equation}
where $\mM[f]$ is the same as in the BGK model. Using the definition (\ref{Max}), (\ref{FP}) can be written equivalently as
\begin{equation}
\mQ(f)=\nabla_v \cdot \left(\nabla_v f+\frac{(v-u)}{T}f \right),
\end{equation}
with $u$ and $T$ given by (\ref{mom}). This is the more commonly seen drift-diffusion type equation in the literature.

\end{itemize}

All of the above collision operators $\mQ$ satisfy the following properties which can be found in many standard textbooks \cite{Cercignani, Villani02} with perhaps the ES-BGK operator as an exception whose proof is given in \cite{ALPP00}.
\begin{itemize}

\item Conservation of mass, momentum, and energy:
\begin{equation}
\langle \mQ(f) \phi \rangle  = 0, \quad \langle \, \cdot \, \phi \rangle :=\int_{\mathbb{R}^d}\cdot \,\phi\,\rd{v}, \quad \phi(v)=\left(1,v,\frac{|v|^2}{2}\right)^T,
\end{equation}
for any function $f$. 

This implies that $\exp(s\mQ)g$, the solution to the homogeneous equation (\ref{homo_ode}) at $t=t_0+s$ with initial data $f|_{t=t_0}=g$, satisfies the conservation property
\begin{equation} \label{cons}
\langle (\exp(s\mQ)g) \phi \rangle = \langle g \phi \rangle ,\quad \forall s\ge 0.
\end{equation}

\item Decay of entropy
\begin{equation} \label{ineq1}
\int_{\mathbb{R}^d}\mQ(f)\log f\,\rd{v}\leq 0,
\end{equation}
further,
\begin{equation}
\int_{\mathbb{R}^d}\mQ(f)\log f\,\rd{v}=0  \Longleftrightarrow \mQ(f)=0 \ \Longleftrightarrow f=\mM[f],
\end{equation}
where $\mM[f]$ is the Maxwellian defined in (\ref{Max}).

This implies that $\exp(s\mQ)g$, the solution to the homogeneous equation (\ref{homo_ode}) at $t=t_0+s$ with initial data $f|_{t=t_0}=g$ has the long time behavior
\begin{equation}\label{longtime}
\lim_{s\rightarrow\infty} \exp(s\mQ)g = \mM[g],
\end{equation}
i.e., $\exp(s\mQ)g$ approaches the Maxwellian determined by the moments of the initial condition.
\end{itemize}
Using these properties, it is easy to show that the spatially inhomogeneous equation (\ref{KE}) has the compressible Euler equations as the leading-order asymptotics when $\varepsilon \rightarrow 0$. Indeed, taking the moments $\langle\, \cdot\, \phi \rangle$ on both sides of (\ref{KE}), one obtains
\begin{equation} \label{localcon}
\partial_t \langle f\phi \rangle + \nabla_x \cdot \langle f v \phi \rangle=0
\end{equation}
by the conservation property of $\mQ$. On the other hand, when $\varepsilon \rightarrow 0$, (\ref{KE}) formally implies $\mQ(f)\rightarrow 0$, hence $f\rightarrow \mM[f]$. Substituting $f=\mM[f]:=\mM[U]$ into (\ref{localcon}) yields 
\begin{equation} \label{Euler1}
\partial_t U + \nabla_x \cdot \langle \mM[U] v \phi \rangle=0,
\end{equation}
where we used the vector $U$ to denote the first $d+2$ moments of $f$: $U=(\rho,\rho u,E)^T$ with $E=\frac{1}{2}\rho u^2+\frac{d}{2}\rho T$ being the total energy. The closed system (\ref{Euler1}) is nothing but the compressible Euler equations
\begin{align} \label{Euler}
\left\{
\begin{array}{l}
\displaystyle \partial_t \rho+\nabla_x\cdot (\rho u)=0,\\[8pt]
\displaystyle \partial_t (\rho u)+\nabla_x\cdot (\rho u\otimes u +pI)=0,\\[8pt]
\displaystyle \partial_t E+\nabla_x\cdot ((E+p)u)=0,
\end{array}\right.
\end{align}
where $p=\rho T$ is the pressure.

We now prove the AP property of the proposed scheme.
\begin{proposition}  \label{prop_AP}
The scheme (\ref{scheme}) (with coefficients satisfying (\ref{coef}) (\ref{coef1}) (\ref{coef2})) applied to the stiff kinetic equation (\ref{ode}) with the collision operator $\mQ$ being the BGK operator (\ref{BGK}), the ES-BGK operator (\ref{ESBGK}), the Boltzmann collision operator (\ref{Boltz}), and the kinetic Fokker-Planck operator (\ref{FP}) is asymptotic-preserving, i.e., for any initial data and fixed $\Delta t$, in the limit $\varepsilon \rightarrow 0$, (\ref{scheme}) becomes a second-order scheme SSP-RK2 applied to the limiting Euler system (\ref{Euler}). Furthermore,
\begin{equation}
\lim_{\varepsilon \rightarrow 0} f^{n+1}=  \mM[U^{n+1}],
\end{equation}
i.e., after each time step, $f^{n+1}$ is driven to its corresponding Maxwellian.
\end{proposition}

\begin{proof}
First of all, taking the moments $\langle \, \cdot \, \phi\rangle$ on (\ref{scheme}) and using (\ref{cons}), one obtains
\begin{equation} \label{mom1}
\begin{split}
& U^{(0)}=U^n,\\
& U^{(1)} = U^{(0)} + b_1 \Delta t \langle \mT (f^{(0)}) \phi \rangle, \\
& U^{(2)} = U^{(1)} + b_2 \Delta t \langle\mT (f^{(1)})\phi \rangle, \\
& U^{n+1} = w U^{(2)} + (1-w) U^n.
\end{split}
\end{equation}
On the other hand, for $a_0,a_1,a_2>0$, using (\ref{longtime}), it can be seen from (\ref{scheme}) that as $\varepsilon\rightarrow 0$, $f^{(0)}$, $f^{(1)}$, and $f^{n+1}$ are driven to their corresponding Maxwellian:
\begin{equation}
\begin{split}
& f^{(0)} \rightarrow \mM[U^n]=\mM[U^{(0)}],\\
& f^{(1)} \rightarrow \mM[U^{(0)} + b_1 \Delta t \langle \mT (f^{(0)}) \phi \rangle]=\mM[U^{(1)}], \\
& f^{n+1} \rightarrow \mM[w U^{(2)} + (1-w) U^n]=\mM[U^{n+1}].
\end{split}
\end{equation}
Finally, substituting $f^{(0)}$ and $f^{(1)}$ into (\ref{mom1}), one has
\begin{equation}\label{moments}
\begin{split}
& U^{(1)} = U^n + b_1 \Delta t \langle\mT (M[U^n])\phi \rangle, \\
& U^{(2)} = U^{(1)} + b_2 \Delta t \langle\mT (M[U^{(1)}] )\phi \rangle, \\
& U^{n+1} = w U^{(2)} + (1-w) U^n.
\end{split}
\end{equation}
With the coefficients (\ref{coef}) and $\mT$ a discretized operator for $-v\cdot \nabla_x$, this is just a kinetic scheme for the limiting Euler equations (\ref{Euler1}) using the SSP-RK2 time discretization.
\end{proof}

\begin{remark}
Note that the requirement for nonzero $a_1$, $a_2$, $a_3$ plays an important role here. In order for the scheme to have a nice AP property (works for any initial data, drives $f$ to the corresponding Maxwellian after each time step, the limiting scheme maintains second-order accuracy, etc.), we need all these coefficients to be non-degenerate. See also the discussion in Section~\ref{subsec:com}.
\end{remark}

\subsection{A slightly different application}

A slightly different example which does not fit exactly into the above framework is the Vlasov-Poisson-Fokker-Planck (VPFP) system, a kinetic description of the Brownian motion of a large system of particles in a surrounding bath. Since it can also be treated using the proposed method, we briefly describe it in this subsection.

The system in the high-field regime reads \cite{Poupaud92}
\begin{equation}\label{VPFP}
\partial_t f + v\cdot\nabla_x f - \frac{1}{\varepsilon}\nabla_x\psi\cdot\nabla_v f =  \frac{1}{\varepsilon}\nabla_v\cdot (\nabla_v f + vf),
\end{equation}
with the electric potential $\psi=\psi(t,x)$ solving the Poisson equation
\begin{equation}\label{Poisson}
-\Delta_x \psi = \rho - h,
\end{equation}
where $\rho=\langle f \rangle $ is the density, and $h=h(x)$ is the background charge density satisfying the neutrality condition $
\int_{\mathbb{R}^d} \rho(0,x) \rd{x} = \int_{\mathbb{R}^d}h(x)\rd{x}$. 

One can write (\ref{VPFP}) as
\begin{equation} \label{VPFP1}
\partial_t f + v\cdot\nabla_x f = \frac{1}{\varepsilon}\mQ(f),\quad \mQ(f) = \nabla_v\cdot (\nabla_v f + (v+\nabla_x\psi)f).
\end{equation}
Here $\mQ(f)$ is a Fokker-Planck type operator, and can be written in the form
\begin{equation}\label{FP1}
\mQ(f) = \nabla_v\cdot \left(M\nabla_v \frac{f}{M}\right),\quad M(v) = M[\psi](v) = \frac{1}{(2\pi)^{d/2}}\exp\left(-\frac{(v+\nabla_x\psi)^2}{2}\right).
\end{equation}
This $\mQ$ satisfies the mass conservation $\langle \mQ(f)\rangle = 0$. Also, $\mQ(f)=0 \Longleftrightarrow f=\rho M[\psi]$.

Now taking the moment $\langle \,\cdot\, \rangle$ on both sides of (\ref{VPFP1}), one has
\begin{equation} 
\partial_t \rho +\nabla_x \cdot \langle v f\rangle=0.
\end{equation}
On the other hand, as $\varepsilon\rightarrow 0$ in (\ref{VPFP1}), formally $\mQ(f)\rightarrow 0$, hence $f\rightarrow \rho M[\psi]$. Substituting this into the above equation, one obtains the limiting equation
\begin{equation}\label{VPFP_limit}
\partial_t \rho - \nabla_x\cdot(\rho \nabla_x\psi) = 0.
\end{equation}

Starting with the form (\ref{VPFP1}), it is easy to see that the scheme (\ref{scheme}) can be applied directly and all the previous discussion regarding the AP property carries over straightforwardly. We omit the detail.

\subsection{Comparison with existing methods}
\label{subsec:com}

Searching the literature, there have been several methods available to solve the stiff kinetic equation (\ref{KE}) or equations of a similar structure. Therefore, we devote this subsection to a careful comparison of our method with some of the existing methods. For a general discussion on exponential integrators, the readers are referred to the review article \cite{MO10}.

\begin{itemize}

\item If one replaces the solution operator $\exp(s\mQ)$ by any second-order approximation, then (\ref{scheme}) remains second order. In particular, if $\exp(s\mQ)$ can be approximated by a second (or higher) order positivity-preserving and AP solver, then one can replace $\exp(s\mQ)$ in (\ref{scheme}) with this solver and still maintains the second-order accuracy, positivity, and AP property. For example, the scheme
\begin{equation}
\begin{split}
& g^{(1)} = g + s\mQ(g^{(1)}), \\
& g^1 = g^{(1)} - \frac{1}{2}s^2\mQ'(g^{(1)})\mQ(g^1),
\end{split}
\end{equation}
produces $g^1$ at $t=t_0+s$, which is a second-order positivity-preserving AP approximation to the exact solution $\exp(s\mQ)g$, in the case of the BGK operator. Using this in (\ref{scheme}) would give an IMEX Runge-Kutta method with correction, similar to our previous work \cite{HSZ18}.

%Also, replacing $\exp(s\mQ)$ by a second-order implicit RK scheme would give a second-order IMEX Runge-Kutta scheme (which is no longer positivity-preserving since there is no implicit SSP scheme of order higher than one [Gottlieb-Shu-Tadmor]). Such implicit RK scheme can be obtained by setting the explicit table to zero in any second order IMEX scheme.

\item
The following two existing second-order methods for (\ref{ode}) are special cases of (\ref{scheme}):
\begin{enumerate}
\item If one considers the second-order Strang splitting
\begin{equation}
\exp(\Delta t (\mT+\mQ)) = \exp\left(\frac{\Delta t}{2}\mQ\right)\exp\left(\Delta t\mT\right)\exp\left(\frac{\Delta t}{2}\mQ\right) + O(\Delta t^3),
\end{equation}
and discretizes $\exp(\Delta t\mT)$ by SSP-RK2, then one arrives at (\ref{scheme}) with 
\begin{equation}
a_0=a_2=\frac{1}{2},\quad a_1=0,\quad b_1=b_2=1,\quad w=\frac{1}{2}.
\end{equation}

\item For the case $\mQ(f)=-\mu f$ with $\mu>0$ a constant, \cite{HS18} rewrites (\ref{ode}) as
\begin{equation}\label{ode_rewrite}
\frac{\rd{}}{\rd{t}}\left(\exp\left(-\frac{t}{\varepsilon}\mQ\right)f\right) = \exp\left(-\frac{t}{\varepsilon}\mQ\right)\mT (f),
\end{equation}
and applies SSP-RK2 to (\ref{ode_rewrite}) directly. Then one arrives at (\ref{scheme}) with 
\begin{equation}
a_0=a_2=0,\quad a_1=1,\quad b_1=b_2=1,\quad w=\frac{1}{2}.
\end{equation}

\end{enumerate}

These two methods are not AP or suffer from order degeneracy in the fluid regime. In fact, in the first method $a_1=0$ and thus $f^{(1)}$ is not at local Maxwellian. Therefore, the flux term $b_2\Delta t \langle \mT(f^{(1)}) \phi \rangle$ in (\ref{mom1}) only approximates the flux in the limiting system up to first-order accuracy, which makes the limiting scheme first order. This order degeneracy of the Strang spliting was discovered already in an early work \cite{Jin95}. Similarly in the second method $a_0=a_2=0$ and thus $f^{(0)} = f^n$ is not at local Maxwellian, which means the flux term $b_1\Delta t \langle \mT(f^{(0)}) \phi \rangle$ in (\ref{mom1}) is only first-order accurate in the limiting scheme. Moreover, even one starts with a consistent initial data, i.e., $f^n=\mM[f^n]$, this method will not drive $f^{n+1}$ to the local Maxwellian since $a_2=0$. Hence this error will pollute the solution as well in the next time step.

For the second method, \cite{HS18} showed that the limiting scheme is second order with consistent initial data, in the case of $\mQ(f)=-\mu f$ and $\mT$ satisfying a maximum principle. Their proof is based on the following fact: if $f$ is at local equilibrium (say $f-f^{\text{eq}} =O(\varepsilon)$), then $f+\Delta t \mT(f)$ is also at local equilibrium ($(f+\Delta t \mT(f))-(f^{\text{eq}}+\Delta t \mT(f^{\text{eq}}))=O(\varepsilon)$). This is clearly not the case for equation (\ref{KE}), since generally speaking $f-\Delta t v\cdot\nabla_x f$ is $O(\Delta t)$ away from its local Maxwellian, even if $f$ itself is at local Maxwellian.  

\item In \cite{DP11}, an exponential Runge-Kutta method was proposed for the homogeneous Boltzmann equation. This method is high order, AP, and positivity-preserving. But it is extended to the non-homogeneous equation (\ref{KE}) based on the Strang-splitting, hence suffers from the order degeneracy as mentioned above. 

\item A non-splitting version of the exponential Runge-Kutta method was proposed in \cite{LP14} by applying an explicit Runge-Kutta scheme to a reformulated spatially inhomogeneous Boltzmann equation. There are two types of schemes proposed. One uses the time varying Maxwellian (called `ExpRK-V' in the paper) which cannot guarantee the positivity of $f$ except the density $\rho$. The other one is based on a fixed Maxwellian (called `ExpRK-F' in the paper) and can preserve the positivity of $f$ provided a separate fluid equation is solved simultaneously and the underlying Runge-Kutta scheme satisfies certain conditions. However, the existence of such schemes (second or third order) that satisfy these conditions as well as AP remains to be discovered. Indeed, the second-order midpoint method and third-order Heun's method cannot satisfy these conditions, unlike what was claimed in \cite{LP14}.

%Compared to this scheme, the main advantages of our scheme are: (1) positivity-preserving can be achieved without the usage of any fluid equation solver; (2) our scheme can be applied to any kinetic equation (not restricted to BGK or Boltzmann type equations), as long as a nice solver for the homogeneous equation is available (see next section for details).

\end{itemize}

To summarize, by a careful choice of the coefficients (\ref{coef}) (\ref{coef1}) (\ref{coef2}), our scheme (\ref{scheme}) is different from any existing exponential Runge-Kutta type methods. It is second order, positivity-preserving, and AP (capturing the Euler limit with second-order accuracy for any initial data).

%%%%%%%%%%%%%%%%%%%%%%%%%%%%%%%%%%%%%%%%%%%%%%%%

\section{Solving the homogeneous equation}
\label{sec:homo}

A key assumption we made in Section~\ref{sec:new} is that the solution to the homogeneous equation (\ref{homo_ode}), or equivalently the solution operator $\exp(s\mQ)$, can be found exactly. From the previous section we have also seen that this can be relaxed by finding an approximate solution, or an approximate operator $\widetilde{\exp}(s \mQ)$, such that it is at least second-order accurate in time, i.e.,
\begin{equation}
\exp(s\mQ)g  \approx \widetilde{\exp}(s \mQ) g + O(s^3);
\end{equation}
positivity preserving, i.e.,
\begin{equation}
 g\ge 0 \, \Longrightarrow \,\widetilde{\exp}(s \mQ) g \ge 0, \quad \forall \ \text{constant} \  s\ge 0;
\end{equation}
and AP, for which to hold we need $\widetilde{\exp}(s \mQ)$ satisfy the same long time behavior as $\exp(s\mQ)$, i.e.,
\begin{equation} \label{longtime1}
\lim_{s\rightarrow\infty} \widetilde{\exp}(s \mQ) g = \mM[g]. 
\end{equation}

In the following, we will provide the strategy to construct $\exp(s\mQ)$ or $\widetilde{\exp}(s \mQ)$ for all the kinetic equations discussed in Section~\ref{sec:AP}.

\subsection{The BGK equation}

For the homogeneous BGK equation
\begin{equation}
\partial_t f=\mQ(f) = \eta(\mM[f] - f), \quad f|_{t=t_0} = g,
\end{equation}
since $\mQ$ conserves mass, momentum and energy, $\mM[f]=\mM[g]$ does not change with time, neither does $\eta$. Hence the solution at $t=t_0+s$ can be found analytically:
\begin{equation}\label{BGK_explicit}
\exp(s\mQ) g = e^{-\eta s}g + (1-e^{-\eta s})\mM[g].
\end{equation}

\subsection{The ES-BGK equation}

For the homogeneous ES-BGK equation
\begin{equation} \label{homo_ESBGK}
\partial_t f= \mQ(f) = \eta(\mG[f]-f), \quad f|_{t=t_0} = g,
\end{equation}
since $\mQ$ conserves mass, momentum and energy, $\rho,u,T$ do not change with time, neither does $\eta$. Taking the moment $\langle \,\cdot\, \frac{1}{\rho}(v-u)\otimes (v-u) \rangle $ on both sides of (\ref{homo_ESBGK}) gives
\begin{equation}
\partial_t \Theta = \eta\left(\frac{1}{\rho}\langle (v-u)\otimes (v-u) \mG[f]\rangle - \Theta\right) = \eta (\bar{T}-\Theta)= \eta (1-\nu)(TI-\Theta),
\end{equation}
whose solution is given by
\begin{equation}\label{Theta_t}
\Theta(t_0+s) = e^{-\eta(1-\nu)s}\Theta(t_0) + (1-e^{-\eta(1-\nu)s}) TI.
\end{equation}
Hence
\begin{equation} \label{T0}
\bar{T}(t_0+s) = \nu e^{-\eta(1-\nu)s}\Theta(t_0) + (1-\nu e^{-\eta(1-\nu)s}) TI.
\end{equation}
On the other hand, (\ref{homo_ESBGK}) can be integrated to yield
\begin{equation}\label{homo_ode_ESBGKs1}
\exp(s\mQ)g=f(t_0+s) = e^{-\eta s}g + \int_{t_0}^{t_0+s} \eta e^{-\eta (t_0+s-\tau)}\mG[f(\tau)]\,\rd{\tau},
\end{equation}
where $\mG[f(\tau)]$ only depends on $\rho,u,\bar{T}(\tau)$. Rather than solving (\ref{homo_ode_ESBGKs1}) exactly, we propose to use a quadrature to approximate the integral part. We adopt the two-point Gauss-Lobatto quadrature, that is,
\begin{equation} \label{GL}
\int_{t_0}^{t_0+s} \eta e^{-\eta (t_0+s-\tau)}\psi(\tau)\,\rd{\tau} \approx w_1\psi(t_0) + w_2\psi(t_0+s),
\end{equation}
where the weights $w_1,w_2$ are determined by requiring this approximation to be exact for $\psi(\tau) = 1,\tau$. A simple calculation gives
\begin{equation}
w_1 = \frac{1-e^{-\eta s}}{\eta s}-e^{-\eta s},\quad w_2 = 1-\frac{1-e^{-\eta s}}{\eta s}.
\end{equation}
The quadrature in (\ref{GL}) has an error $O(s^3)$ for general functions.

Therefore, we approximate the solution in (\ref{homo_ode_ESBGKs1}) as
\begin{equation}\label{homo_ode_ESBGKs2}
\exp(s\mQ)g  \approx \widetilde{\exp}(s \mQ) g:=e^{-\eta s}g + \left(\frac{1-e^{-\eta s}}{\eta s}-e^{-\eta s}\right)\mG[\rho,u,\bar{T}(t_0)] + \left(1-\frac{1-e^{-\eta s}}{\eta s}\right)\mG[\rho,u,\bar{T}(t_0+s)],
\end{equation}
with $\bar{T}(t_0+s)$ given by (\ref{T0}). This approximation is positivity-preserving since (\ref{homo_ode_ESBGKs2}) is a convex combination of positive functions. The resulting scheme is AP because $\widetilde{\exp}(s \mQ) g$ satisfies the long time behavior (\ref{longtime1}): as $s\rightarrow\infty$, one has $\bar{T}(t_0+s)\rightarrow TI$, and thus $\mG[\rho,u,\bar{T}(t_0+s)] \rightarrow \mM[g]$. Also, three weights in (\ref{homo_ode_ESBGKs2}) converge to 0, 0, 1, respectively, hence $\widetilde{\exp}(s \mQ) g \rightarrow \mM[g]$.

\subsection{The Boltzmann equation}

For the homogeneous Boltzmann equation 
\begin{equation}\label{homo_ode_B}
\partial_t f =\mQ(f)=\int_{\mathbb{R}^d}\int_{S^{d-1}}B(v-v_*, \sigma) [f(v')f(v_*')-f(v)f(v_*)]\,\rd{\sigma}\,\rd{v_*}, \quad f|_{t=t_0} = g,
\end{equation}
we adopt the exponential Runge-Kutta method introduced in \cite{DP11} to find an approximate solution. Since $\mQ$ conserves mass, momentum and energy, $\mM[f]=\mM[g]$ does not change with time. Thus we can rewrite (\ref{homo_ode_B}) as
\begin{equation}\label{homo_ode_B1}
\partial_t ((f-\mM)e^{\mu t}) = (P(f) - \mu \mM) e^{\mu t},
\end{equation}
where $P(f) := \mQ(f) + \mu f$, $\mu>0$ being a constant, large enough so that $P(f)\geq 0$ (a simple choice is $\mu=\sup_v\int_{\mathbb{R}^d}\int_{S^{d-1}}B(v-v_*, \sigma) f(v_*)\,\rd{\sigma}\,\rd{v_*}$). Then, by applying the midpoint method to (\ref{homo_ode_B1}), one obtains a second-order scheme
\begin{equation}
\begin{split}
& (f^{(1)}-\mM) e^{\frac{\lambda}{2}} = (g-\mM) + \frac{\lambda}{2}\left(\frac{P(g)}{\mu}-\mM\right), \\
& (f^1-\mM) e^{\lambda} = (g-\mM) + \lambda e^{\frac{\lambda}{2}}\left(\frac{P(f^{(1)})}{\mu}-\mM\right),
\end{split}
\end{equation}
with $\lambda=\mu s$, which simplifies to
\begin{equation}\label{Boltz_sch}
\begin{split}
& f^{(1)} = e^{-\frac{\lambda}{2}}g + \left(1-e^{-\frac{\lambda}{2}} - \frac{\lambda}{2}e^{-\frac{\lambda}{2}}\right)\mM + \frac{\lambda}{2}e^{-\frac{\lambda}{2}} \frac{P(g)}{\mu}, \\
&  f^1 = e^{-\lambda}g + \left(1-e^{-\lambda} - \lambda e^{-\frac{\lambda}{2}}\right)\mM + \lambda e^{-\frac{\lambda}{2}} \frac{P(f^{(1)})}{\mu}.
\end{split}
\end{equation}
Therefore, we choose $f^1$ to be the approximate solution at $t=t_0+s$:
\begin{equation}
\exp(s\mQ)g \approx  \widetilde{\exp}(s \mQ) g:=f^1.
\end{equation}
This approximation is positivity-preserving since both $f^{(1)}$ and $f^1$ are convex combinations of positive functions. It is AP since $s\rightarrow\infty$ implies $\lambda\rightarrow\infty$, thus $f^1\rightarrow \mM$.

\begin{remark}
\label{rmk:Boltz}
Here we did not address the issue of velocity domain discretization. To get a fully discrete scheme, one also needs an efficient and positivity-preserving solver for the Boltzmann collision operator (to evaluate the term $P(f)$ in the scheme). Available choices are the direct simulation Monte Carlo (DSMC) method \cite{Bird}, the discrete velocity method \cite{MPR13}, or the recently proposed entropic Fourier method \cite{CFY18}.
\end{remark}

\subsection{The kinetic Fokker-Planck equation}

For the homogeneous kinetic Fokker-Planck equation
\begin{equation}\label{homo_ode_FP}
\partial_t f = \mQ(f) =  \nabla_v\cdot \left(\mM[f]\nabla_v \frac{f}{\mM[f]}\right),\quad f|_{t=t_0} = g,
\end{equation}
since $\mQ$ conserves mass, momentum and energy, $\mM[f]=\mM[g]$ does not change with time. We adopt the approximation proposed in \cite{JY11} to discretize $\mQ$. Define $\tilde{f}=\frac{f}{\sqrt{\mM}}$, then $\tilde{f}$ solves
\begin{equation}\label{homo_ode_FP1}
\partial_t \tilde{f} = \tilde{\mQ}(\tilde{f}) := \frac{1}{\sqrt{\mM}}\nabla_v\cdot (\mM\nabla_v \frac{\tilde{f}}{\sqrt{\mM}}),\quad \tilde{f}|_{t=t_0}=\tilde{g}:=\frac{g}{\sqrt{\mM}}.
\end{equation}
Hence
\begin{equation}
\exp(s\mQ)g = \sqrt{\mM}\exp(s\tilde{\mQ})\tilde{g}.
\end{equation}
Now it suffices to approximate $\exp(s\tilde{\mQ})\tilde{g}$. To do this, we will first discretize the velocity and then use the matrix exponential to solve (\ref{homo_ode_FP1}). For simplicity we consider $v$ in 1d. We truncate the velocity domain into a large enough interval $[-|v|_{\text{max}},|v|_{\text{max}}]$ and discretize it into $N_v$ grid points with $v_i = -|v|_{\text{max}}+(i-1/2)\Delta v,\,i=1,\dots,N_v$, $\Delta v = 2|v|_{\text{max}}/N_v$. Then the operator $\tilde{\mQ}$ can be approximated by a tridiagonal symmetric matrix $\tilde{\mQ}^h$ with the entries given by
\begin{equation}\label{tildeQ}
\begin{split}
& \tilde{\mQ}^h_{i,i} =  - \frac{1}{\Delta v^2}\frac{\sqrt{\mM_{i-1}}+\sqrt{\mM_{i+1}}}{\sqrt{\mM_i}},\\
& \tilde{\mQ}^h_{i,i-1} =  \tilde{\mQ}^h_{i,i+1} =  \frac{1}{\Delta v^2}, \\
\end{split}
\end{equation}
where $\mM_i=\mM(v_i)$. Define the vector $\tilde{g}^h=(\tilde{g}_1,\dots,\tilde{g}_{N_v})^T$ with $\tilde{g}_i=\tilde{g}(v_i)$, then we approximate $\exp(s\mQ)g$ as
\begin{equation} \label{matrixexp}
\left(\exp(s\mQ)g\right)_i \approx \sqrt{\mM_i}\left(\exp(s\tilde{\mQ}^h) \tilde{g}^h\right)_i,
\end{equation}
where $\exp(s\tilde{\mQ}^h)\tilde{g}^h$ can be computed very accurately by existing matrix exponential algorithms (for simplicity we assume there is no error occurring at this step). This approximation is positivity-preserving since the off-diagonal entries of $\tilde{\mQ}^h$ are non-negative. It is AP since $s\rightarrow\infty$ implies $\sqrt{\mM_i}\left (\exp(s\tilde{\mQ}^h) \tilde{g}^h\right)_i \rightarrow\mM_i$. To see this, note that the discretization (\ref{tildeQ}) for equation (\ref{homo_ode_FP1}) is equivalent to the following
\begin{equation}
\partial_t f_i=\frac{F_{i+1/2}-F_{i-1/2}}{\Delta v}, \quad F_{i+1/2}:=\frac{\sqrt{\mM_i\mM_{i+1}}}{\Delta v}\left(\frac{f_{i+1}}{\mM_{i+1}}-\frac{f_i}{\mM_i}\right).
\end{equation}
Define the discrete relative entropy as
\begin{equation}
H=\sum_i f_i\log \frac{f_i}{\mM_i}\Delta v,
\end{equation}
then
\begin{equation}
\begin{split}
\partial_t H&=\sum_i\partial_t f_i \left( \log\frac{f_i}{\mM_i}+1\right)\Delta v=\sum_i (F_{i+1/2}-F_{i-1/2})  \left( \log\frac{f_i}{\mM_i}+1\right)\\
&=-\sum_iF_{i+1/2}\left(\log\frac{f_{i+1}}{\mM_{i+1}}-\log\frac{f_i}{\mM_i}\right)\\
&=-\sum_i\frac{\sqrt{\mM_i\mM_{i+1}}}{\Delta v}\left(\frac{f_{i+1}}{\mM_{i+1}}-\frac{f_i}{\mM_i}\right)\left(\log\frac{f_{i+1}}{\mM_{i+1}}-\log\frac{f_i}{\mM_i}\right)\leq 0,
\end{split}
\end{equation}
and the equality holds if and only if $f_i/\mM_i$ is independent of $i$. This implies $f_i=\mM_i$ by conservation.

We remark that for the VPFP system, one can solve the homogeneous equation (\ref{homo_ode}) by the same method, since for the homogeneous equation the $\psi$ appeared in (\ref{FP1}) does not change in time.

%%%%%%%%%%%%%%%%%%%%%%%%%%%%%%%%%%%%%%%%%%%%%%%%

\section{Entropy-decay property}
\label{sec:entropy}

In this section, we discuss the entropy-decay property of our scheme. First of all, we recall the following well-known result in kinetic theory. For the kinetic equation (\ref{KE}) with the collision operator being the BGK operator (\ref{BGK}), the ES-BGK operator (\ref{ESBGK}), the Boltzmann collision operator (\ref{Boltz}), or the kinetic Fokker-Planck operator (\ref{FP}), one has
\begin{equation} \label{entropy}
\frac{\rd{}}{\rd{t}} \iint f\log f\,\rd{v}\,\rd{x}\leq 0
\end{equation}
under a periodic or compactly supported boundary condition in $x$. This is the famous H-theorem which says that the total entropy of the system is always non-increasing. %The key to show (\ref{entropy}) is to prove the inequality
%\begin{equation} \label{ineq1}
%\int \mQ(f) \log f\,\rd{v}\leq 0,
%\end{equation}
%which can be found in many standard textbooks \cite{Cercignani, Villani02} with perhaps the ES-BGK operator as an exception whose proof is given in \cite{ALPP00}. %Note that (\ref{ineq1}) is equivalent of saying for the homogeneous equation (\ref{homo_ode})
%\begin{equation} 
%S[\exp(s\mQ)g]\leq S[g], \quad \forall \ \text{constant} \  s\ge 0, \quad S[f]:=\int f \log f\,\rd{v}.
%\end{equation}

We would like to show that our scheme (\ref{scheme}) coupled with the first-order upwind discretization for the transport term and the homogeneous solvers discussed in Section~\ref{sec:homo} satisfies a discrete entropy-decay property (a discrete analog of (\ref{entropy})). In order to do so, we assume the velocity space is continuous, in particular, this means the Fokker-Planck operator is not discretized and the solution to its homogeneous equation can be found analytically.

For simplicity, we consider the equation (\ref{KE}) in 1d:
\begin{equation} \label{KE1d}
\partial_t f+ v\partial_x f=\frac{1}{\varepsilon}\mQ(f).
\end{equation}
We truncate the velocity domain to a large enough interval $[-|v|_{\text{max}}, |v|_{\text{max}}]$ and discretize the transport term by the upwind method ($j$ is the spatial index):
\begin{equation} \label{1Dupwind}
(v\partial_x f)_j = \chi_{v\ge 0}v\frac{f_j - f_{j-1}}{\Delta x} + \chi_{v<0}v\frac{f_{j+1} - f_j}{\Delta x}.
\end{equation}
Define the discrete entropy as
\begin{equation} 
\mS[f]:=\Delta x \sum_j S[f_j], \quad S[f_j]:=\int f_j \log f_j\,\rd{v},
\end{equation}
then we claim that the scheme (\ref{scheme}) satisfies a discrete entropy-decay property:
\begin{equation} \label{disentropy}
\mS[f^{n+1}]\leq \mS[f^n].
\end{equation}

To prove (\ref{disentropy}), we need two building blocks, one is the exponential step decays entropy, i.e., for either the exact $\exp(s\mQ)$ or approximate $ \widetilde{\exp}(s \mQ) $, one has
\begin{equation} \label{expentropy}
\mS[\exp(s\mQ)g]\leq \mS[g] \  \ \text{or} \  \ \mS[ \widetilde{\exp}(s \mQ) g]\leq \mS[g], \quad \forall \ \text{constant} \  s\ge 0,
\end{equation}
the other is the transport step decays entropy, i.e., for step of the form $g=f+a \Delta t \mT(f)$, one has
\begin{equation} \label{tranentropy}
\mS[g]\leq \mS[f], \quad \text{under the CFL condition} \ \Delta t\leq \frac{\Delta x}{a|v|_{\text{max}}}.
\end{equation}

We now prove (\ref{expentropy}) and (\ref{tranentropy}), respectively.

\begin{itemize}
\item For the BGK and Fokker-Planck operators, we have the exact $\exp(s\mQ)$, hence $\mS[\exp(s\mQ)g]\leq \mS[g]$ follows directly from the analytical result (\ref{ineq1}). 

For the ES-BGK operator, note that (\ref{homo_ode_ESBGKs2}) is a convex combination of $g$, $\mG[g]$ and $\mG[\exp(s\mQ)g]$. One has $\mS[\mG[g]]\le S[g]$ from \cite{ALPP00}, hence $\mS[\mG[\exp(s\mQ)g]]\le \mS[\exp(s\mQ)g] \le \mS[g]$ (the second inequality comes from the analytical result (\ref{ineq1})). Therefore, $\mS[ \widetilde{\exp}(s \mQ) g]\leq \mS[g]$ follows from  the convexity of $\mS$. 

For the Boltzmann operator, note that in the approximation (\ref{Boltz_sch}), $f^{(1)}$ is a convex combination of $g$, $\mM$ and $\frac{P(g)}{\mu}$, and $f^1$ is a convex combination of $g$, $\mM$ and $\frac{P(f^{(1)})}{\mu}$. In \cite{Villani98}, it is proved that $\mS[\frac{P(f)}{\mu}] \le \,\mS[f]$ (for Maxwell molecules). Therefore, by the convexity of $\mS$ and $\mS[\mM[g]]\leq \mS[g]$, one has $\mS[f^{(1)}]\le \mS[g]$, hence $\mS[f^1]\le \mS[g]$. Therefore, $\mS[ \widetilde{\exp}(s \mQ) g]\leq \mS[g]$.

\item The transport step $g=f+a \Delta t \mT(f)$ with (\ref{1Dupwind}) plugged in reads
\begin{equation}
\begin{split}
g_j&=f_j-a\Delta t \left (\chi_{v\ge 0}v\frac{f_j - f_{j-1}}{\Delta x} + \chi_{v<0}v\frac{f_{j+1} - f_j}{\Delta x}\right)\\
&=\left(1-a\frac{|v|\Delta t}{\Delta x}\right)f_j + a\frac{|v|\Delta t}{\Delta x} \left(\chi_{v\ge0}f_{j-1} + \chi_{v<0}f_{j+1}\right).
\end{split}
\end{equation}
Hence the right hand side is a convex combination of $f_j$ and $\chi_{v\ge0}f_{j-1} + \chi_{v<0}f_{j+1}$ under the CFL condition $\Delta t\leq \frac{\Delta x}{a|v|_{\text{max}}}$. Then using the convexity of function $f\log f$, one has
\begin{equation}
\begin{split} \label{disflux}
S[g_j]\leq &\int \left(1-a\frac{|v|\Delta t}{\Delta x}\right) f_j \log f_j\,\rd{v} \\
& +\int a\frac{|v|\Delta t}{\Delta x} \left(\chi_{v\ge0}f_{j-1}+ \chi_{v<0}f_{j+1}\right) \log \left(\chi_{v\ge0}f_{j-1} + \chi_{v<0}f_{j+1}\right)\,\rd{v}\\
=& S[f_j]-a\frac{\Delta t}{\Delta x} \left( F_{j+1/2}-F_{j-1/2} \right),
\end{split} 
\end{equation}
where
\begin{equation}
F_{j+1/2}:=\int |v|\left(\chi_{v\geq 0} f_j \log f_j - \chi_{v<0}f_{j+1}\log f_{j+1}\right)\,\rd{v}
\end{equation}
is the discrete entropy flux. Summing over $j$ in (\ref{disflux}) and assuming the periodic or compactly supported boundary condition in $x$, one obtains
\begin{equation}
\mS[g]\leq \mS[f].
\end{equation}
\end{itemize}

Now applying the previous two results in (\ref{scheme}), we have
\begin{equation}
\mS[f^{(2)}]\leq \mS[f^{(1)}]\leq \mS[f^{(0)}]\leq \mS[f^n],
\end{equation}
hence
\begin{equation}
\mS[f^{n+1}]\leq w\mS[f^{(2)}]+(1-w)S[f^n]\leq \mS[f^n].
\end{equation}
The assertion is proved.

%%%%%%%%%%%%%%%%%%%%%%%%%%%%%%%%%%%%%%%%%%%%%%%%

\section{A remark on spatial and velocity discretizations}
\label{sec:spatial}

Most of the spatial and velocity discretizations follow our previous paper \cite{HSZ18}, namely, we use a finite volume method for the $x$-variable and finite difference method for the $v$-variable.

For the transport term, we adopt the fifth-order finite volume WENO method \cite{Shu98} with a bound-preserving limiter \cite{ZS10, ZS11} to insure the positivity. Since the treatment of this part is standard and has been described in \cite{HSZ18}, we omit the detail.

For the collision term, special care needs to be paid when switching between the finite volume and finite difference framework. We briefly describe the procedure in the following. For convenience, we regard $v$ as continuous and omit it in the discussion.

Let $I_j=[x_{j-1/2},x_{j+1/2}]$ be the $j$-th spatial cell, and $\{x_{j,l}\} \ (l=1,2,3)$ denote the three Gauss-Legendre quadrature points in this cell and $\{w_l\}$ be the corresponding quadrature weights. For a fixed $v$, suppose we are given the cell average $f_j\geq 0$ in $I_j$, we would like to construct a polynomial $f_j(x)$ of degree four such that
\begin{itemize}
\item $f_j(x)$ is a fifth-order accurate approximation to $f(x)$ in $I_j$ with $f_j$ being its cell average, i.e., 
\begin{equation} \label{poly}
\frac{1}{\Delta x}\int_{x_{j-1/2}}^{x_{j+1/2}} f_j(x)\rd{x} = f_j.
\end{equation}
\item $f_j(x)$ is non-negative at the Gauss quadrature points, i.e.,
\begin{equation}
f_{j,l}:=f_j(x_{j,l})\geq 0, \quad l=1,2,3.
\end{equation}
\end{itemize}
The construction of such a polynomial can be done similarly as described in Section 3.2.2 of our previous paper \cite{HSZ18}. Provided with $f_j(x)$, it is easy to see (\ref{poly}) reduces to
\begin{equation}\label{ave_exact}
\sum_{l=1}^3 w_l f_{j,l} = f_j,
\end{equation}
since the three-point Gauss-Legendre quadrature is exact for polynomials with degree no more than five.

Then we approximate the $j$-th cell average of $\exp(s\mQ)f$ by
\begin{equation}\label{point_exp}
(\exp(s\mQ)f)_j = \sum_{l=1}^3 w_l \exp(s\mQ)f_{j,l}.
\end{equation}
This approximation is fifth-order accurate in $x$ since the reconstruction $f_{j,l}$ is. It is also conservative, since
\begin{equation}
\langle (\exp(s\mQ)f)_j\phi\rangle = \sum_{l=1}^3 w_l \langle\exp(s\mQ)f_{j,l}\phi\rangle= \sum_{l=1}^3 w_l \langle f_{j,l}\phi\rangle = \langle\sum_{l=1}^3 w_l  f_{j,l}\phi\rangle = \langle f_j \phi\rangle,
\end{equation}
where we used (\ref{cons}) in the second equality and (\ref{ave_exact}) in the last one. 

For the mixed regime problem where $\varepsilon=\varepsilon(x)$, one needs to compute $\exp(s(x)\mQ f)$ with $s(x)$ a given function depending on $x$. To do this, we use the same reconstruction $f_{j,l}$ and approximate $\exp(s(x)\mQ f)$ by
\begin{equation}\label{point_exp}
(\exp(s(x)\mQ)f)_j = \sum_{l=1}^3 w_l \exp(s(x_{j,l})\mQ)f_{j,l},
\end{equation}
which is still fifth-order accurate and conservative.

%%%%%%%%%%%%%%%%%%%%%%%%%%%%%%%%%%%%%%%%%%%%%%%%

\section{Numerical examples}
\label{sec:num}

In this section we demonstrate numerically the properties of the proposed scheme (\ref{scheme}) with coefficients (\ref{coef}) (\ref{coef3}). We use the 1d BGK and Fokker-Planck equations as prototype examples since the main purpose of this work is to develop a generic time integrator that can be potentially applied to a large class of equations rather than to study a particular kinetic equation.

We consider the computational domain $x\in [0,2]$ with periodic boundary condition (except the test in Section~\ref{subsec:pos}, where the Dirichlet boundary condition is assumed), and a large enough velocity domain $v\in[-|v|_{\text{max}},|v|_{\text{max}}]$ with $|v|_{\text{max}}=15$. The $x$-space is discretized into $N_x$ cells with $\Delta x = 2/{N_x}$ and cell center $x_j=(j-1/2)\Delta x$, $j=1,\dots,N_x$. The $v$-space is discretized into $N_v$ grid points with $\Delta v = 2|v|_{\text{max}}/{N_v}$ and $v_i=-|v|_{\text{max}}+(i-1/2)\Delta v$, $i=1,\dots,N_v$. Unless specified, $N_v=150$ is used in all tests so that the discretization error in $v$ is much smaller than that in $x$ and $t$. 

To compute the matrix exponential (\ref{matrixexp}) resulting from the discretization of the Fokker Planck operator, we used the code by S. Guttel \cite{Guttel10} for the test in Section~\ref{subsec:acc}, and the MATLAB function `expm' for other tests.

\subsection{Accuracy test}
\label{subsec:acc}

We first verify the second-order accuracy of the scheme. We consider inconsistent initial data
\begin{equation}\label{incon1_2}
f(0,x,v) = 0.5M_{\rho,u,T} + 0.3M_{\rho,-0.5u,T},
\end{equation}
with
\begin{equation}\label{incon1_1}
\rho = 1+0.2\sin(\pi x),\quad u = 1,\quad T = \frac{1}{1+0.2\sin(\pi x)},
\end{equation}
and compute the solution to time $t=0.1$. We choose different values of $\varepsilon$, ranging from the kinetic regime ($\varepsilon=1$) to the fluid regime ($\varepsilon=10^{-10}$). We choose different $\Delta x$ and set $\Delta t = 0.5\Delta x/{|v|_{\text{max}}}$. This CFL number is not small enough to guarantee the positivity which is pretty restrictive due to the spatial discretization. We will consider the positivity-preserving property in the following test. For the same reason, the positivity-preserving limiter is turned off here. Since the exact solution is not available, the numerical solution on a finer mesh $\Delta x/2$ is used as a reference solution to compute the error for the solution on the mesh of size $\Delta x$:
\begin{equation}
\text{error}_{\Delta t,\Delta x}:=\|f_{\Delta t,\Delta x}-f_{\Delta t/2,\Delta x/2}\|_{L^2_{x,v}}.
\end{equation}
The results are shown in Tables 7.1 and 7.2. For the Fokker-Planck equation, due to the second-order discretization error in the velocity space, one has to choose a larger $N_v$ in order to see the temporal error. In all these results, the spatial error dominates for small $N_x$, and the temporal error dominates for large $N_x$. One can clearly see that in both the kinetic regime $\varepsilon=O(1)$ and the fluid regime $\varepsilon\ll 1$, the scheme is second order. Note that there is some extent of order reduction in the intermediate regime $\varepsilon=O(\Delta t)$. The uniform accuracy of the AP scheme is an open problem and we do not attempt to address this issue in the current work.

\begin{table}
\begin{center}
\footnotesize
\begin{tabular}{|l|c|c|c|c|c|c|}
\hline
&$\varepsilon=1e+00$&$\varepsilon=1e-02$&$\varepsilon=1e-04$&$\varepsilon=1e-06$&$\varepsilon=1e-08$&$\varepsilon=1e-10$\\
\hline
Nx=10&5.60e-04&4.64e-04&4.67e-04&4.67e-04&4.67e-04&4.67e-04\\
\hline
Nx=20&5.91e-05&3.93e-05&3.65e-05&3.65e-05&3.65e-05&3.65e-05\\
Order&3.25&3.56&3.68&3.68&3.68&3.68\\
\hline
Nx=40&4.33e-06&2.83e-06&4.46e-06&2.46e-06&2.46e-06&2.46e-06\\
Order&3.77&3.80&3.03&3.89&3.89&3.89\\
\hline
Nx=80&2.11e-07&2.86e-07&5.24e-06&1.10e-07&1.10e-07&1.10e-07\\
Order&4.36&3.31&-0.23&4.49&4.49&4.49\\
\hline
Nx=160&1.27e-08&6.24e-08&3.25e-06&6.29e-09&6.29e-09&6.29e-09\\
Order&4.05&2.19&0.69&4.12&4.12&4.12\\
\hline
Nx=320&2.89e-09&1.55e-08&1.23e-06&1.45e-09&1.45e-09&1.45e-09\\
Order&2.14&2.01&1.40&2.11&2.11&2.11\\
\hline
Nx=640&7.30e-10&3.88e-09&3.74e-07&3.68e-10&3.68e-10&3.68e-10\\
Order&1.99&2.00&1.72&1.98&1.98&1.98\\
\hline
Nx=1280&1.83e-10&9.71e-10&1.03e-07&2.82e-10&9.20e-11&9.20e-11\\
Order&2.00&2.00&1.86&0.38&2.00&2.00\\
\hline
\end{tabular}
\caption{Accuracy test of the scheme for the BGK equation.}
\end{center}
\end{table}

\begin{table}
\begin{center}
\footnotesize
\begin{tabular}{|l|c|c|c|c|c|c|c|c|}
\hline
&$\varepsilon=1e+00$&$\varepsilon=1e-01$&$\varepsilon=1e-02$&$\varepsilon=1e-03$&$\varepsilon=1e-04$&$\varepsilon=1e-05$&$\varepsilon=1e-06$&$\varepsilon=1e-07$\\
\hline
Nx=10&5.30e-04&4.64e-04&4.62e-04&4.66e-04&4.66e-04&4.66e-04&4.66e-04&4.66e-04\\
\hline
Nx=20&5.50e-05&4.32e-05&3.93e-05&4.63e-05&3.65e-05&3.65e-05&3.65e-05&3.65e-05\\
Order&3.27&3.42&3.56&3.33&3.68&3.68&3.68&3.68\\
\hline
Nx=40&3.89e-06&2.82e-06&3.42e-06&1.29e-05&2.54e-06&2.46e-06&2.46e-06&2.46e-06\\
Order&3.82&3.94&3.52&1.85&3.85&3.89&3.89&3.89\\
\hline
Nx=80&1.80e-07&1.29e-07&5.47e-07&4.23e-06&2.16e-06&1.10e-07&1.10e-07&1.10e-07\\
Order&4.43&4.45&2.64&1.61&0.23&4.49&4.49&4.49\\
\hline
Nx=160&1.13e-08&9.34e-09&1.35e-07&1.25e-06&2.97e-06&1.16e-08&6.30e-09&6.29e-09\\
Order&3.99&3.79&2.02&1.76&-0.46&3.25&4.12&4.12\\
\hline
Nx=320&2.64e-09&2.07e-09&3.56e-08&3.53e-07&1.80e-06&1.08e-07&1.50e-09&1.45e-09\\
Order&2.10&2.17&1.92&1.82&0.72&-3.22&2.07&2.11\\
\hline
Nx=640&6.66e-10&5.77e-10&9.93e-09&1.01e-07&7.15e-07&3.91e-07&3.62e-10&3.68e-10\\
Order&1.98&1.85&1.84&1.80&1.33&-1.86&2.05&1.98\\
\hline
\end{tabular}
\caption{Accuracy test of the scheme for the Fokker-Planck equation. Here $N_v=600$.}
\end{center}
\end{table}

\subsection{Positivity-preserving property}
\label{subsec:pos}

We now illustrate the positivity-preserving property of the scheme. Consider the initial data
\begin{equation}
f(0,x,v) = M_{\rho,u,T},
\end{equation}
with
\begin{equation}
(\rho,u,T)=\left\{\begin{split}
& (1,0,1),\quad  0\leq x\leq 1,\\
& (0.125,0,0.25),\quad 1<x \leq 2.
\end{split}\right.
\end{equation}

With the positivity-preserving limiter, the CFL condition of our scheme is $\Delta t\leq \frac{1}{12}\frac{\Delta x}{ |v|_{\text{max}}}$ (note that $1/12$ comes from the spatial discretization and the forward Euler method also has the same constraint). We choose $\Delta t = \frac{1}{24}\frac{\Delta x}{ |v|_{\text{max}}}$ and $N_x=80$. 

For the BGK equation, no negative cells are detected in the simulation. For the Fokker-Planck equation, one technical issue is that we are not aware of any algorithms that can guarantee the numerically computed matrix exponential is positive if the exact matrix exponential is. To demonstrate that no negative values are caused by our time discretization, we use `expm' function in MATLAB to compute the matrix exponential and set the negative entries of the resulting matrix to zero. With this modification, no negative cells are detected in the simulation.

As a comparison, we solve the same equations with the same initial data and spatial/velocity discretization, but using the ARS(2,2,2) scheme in time \cite{ARS97}, which is a standard second-order accurate IMEX scheme with no positivity-preserving property. The number of negative cells during the simulation is tracked. The result for the BGK equation is already included in the previous paper \cite{HSZ18} and is omitted here. The result for the Fokker-Planck equation is shown in Figure~\ref{fig:negcell}. Here to make the comparison fair, when we compute $(I-s\tilde{\mQ}^h)^{-1}g^h$ (an operator needs to be evaluated in the IMEX scheme), we first compute the matrix $(I-s\tilde{\mQ}^h)^{-1}$ which is not necessarily positive at the numerical level, and then set the negative entries to zero in this matrix. This is to make sure that no negative values are generated due to the failure of positivity-preserving in the matrix inversion. In Figure 1 one can still see a lot of negative cells in the fluid regime. 

\begin{figure}
\begin{center}
	\includegraphics[width=6in,height=3.5in]{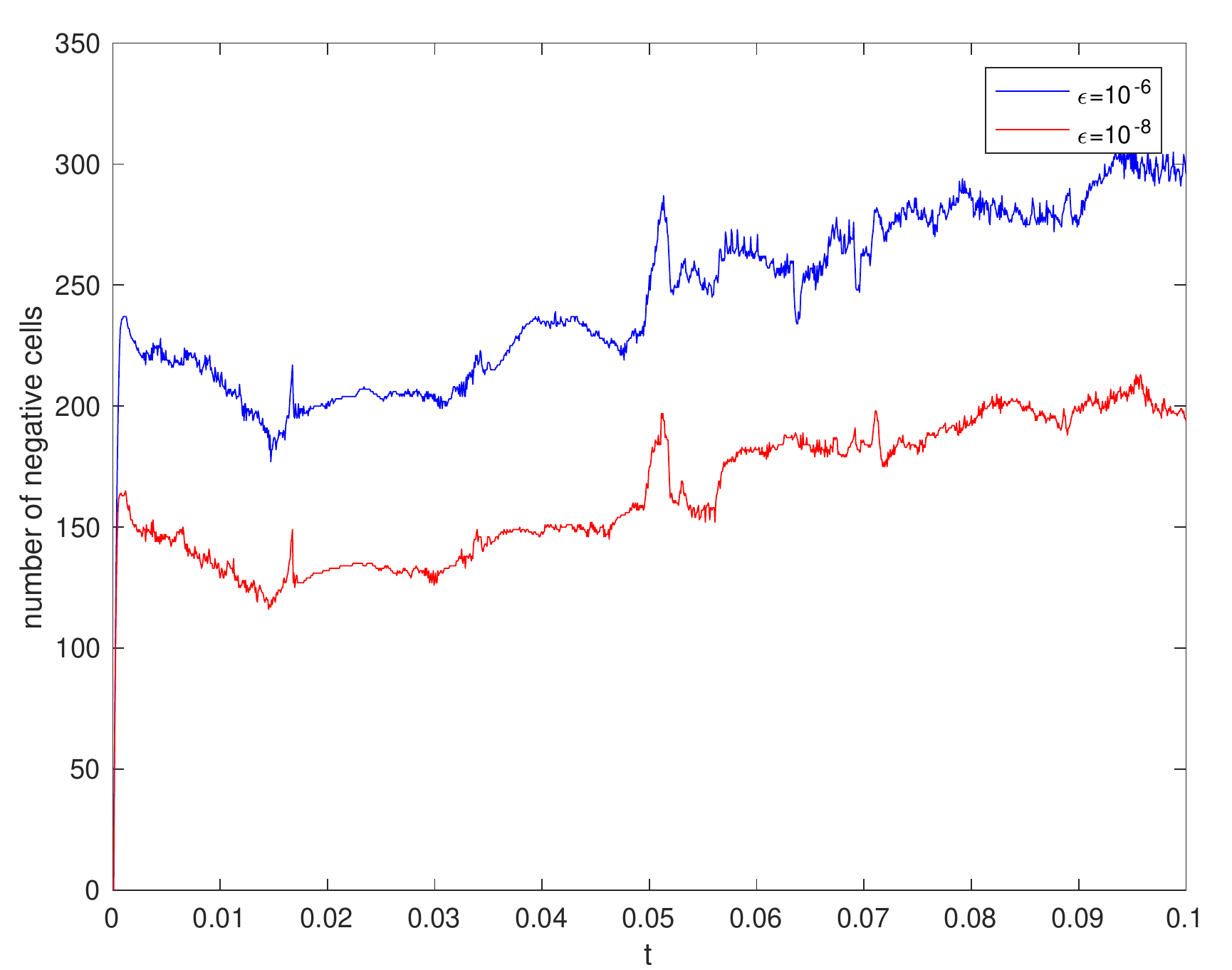}
	\caption{Total number of negative cells for the ARS(2,2,2) scheme applied to the Fokker-Planck equation during time evolution. Blue line: $\varepsilon=10^{-6}$. Red line: $\varepsilon=10^{-8}$.}
	\label{fig:negcell}
\end{center}	
\end{figure}

\subsection{AP property}

Finally, to illustrate the AP property, we use the proposed scheme to solve the BGK and Fokker-Planck equations in a mixed regime ($\varepsilon$ is a function of $x$ so that in part of the domain the problem is in kinetic regime and while in other part it is in fluid regime). We take the same initial data as in (\ref{incon1_2})-(\ref{incon1_1}) and $N_x=40$.

For the BGK equation, we consider $\varepsilon=\varepsilon(x)$ as follows:
\begin{equation}
\varepsilon(x) = \varepsilon_0 + (\tanh(1-11(x-1))+\tanh(1+11(x-1))), \quad \varepsilon_0 = 10^{-5}.
\end{equation}
We compare the macroscopic quantities at time $t=0.5$ with a reference solution computed by SSP-RK2 with $N_x=80$. Note that for our scheme, $\Delta t = \frac{1}{24}\frac{\Delta x}{|v|_{\text{max}}}\approx 7\times 10^{-5}$; while for the explicit SSP-RK2 scheme, $\Delta t = \frac{1}{240}\frac{\Delta x}{|v|_{\text{max}}}\approx 7\times 10^{-6} $ which needs to resolve $\varepsilon$. One can see a good agreement with the reference solution in Figure \ref{fig:APBGK}.

\begin{figure}
\begin{center}
	\includegraphics[width=5.6in,height=2.5in]{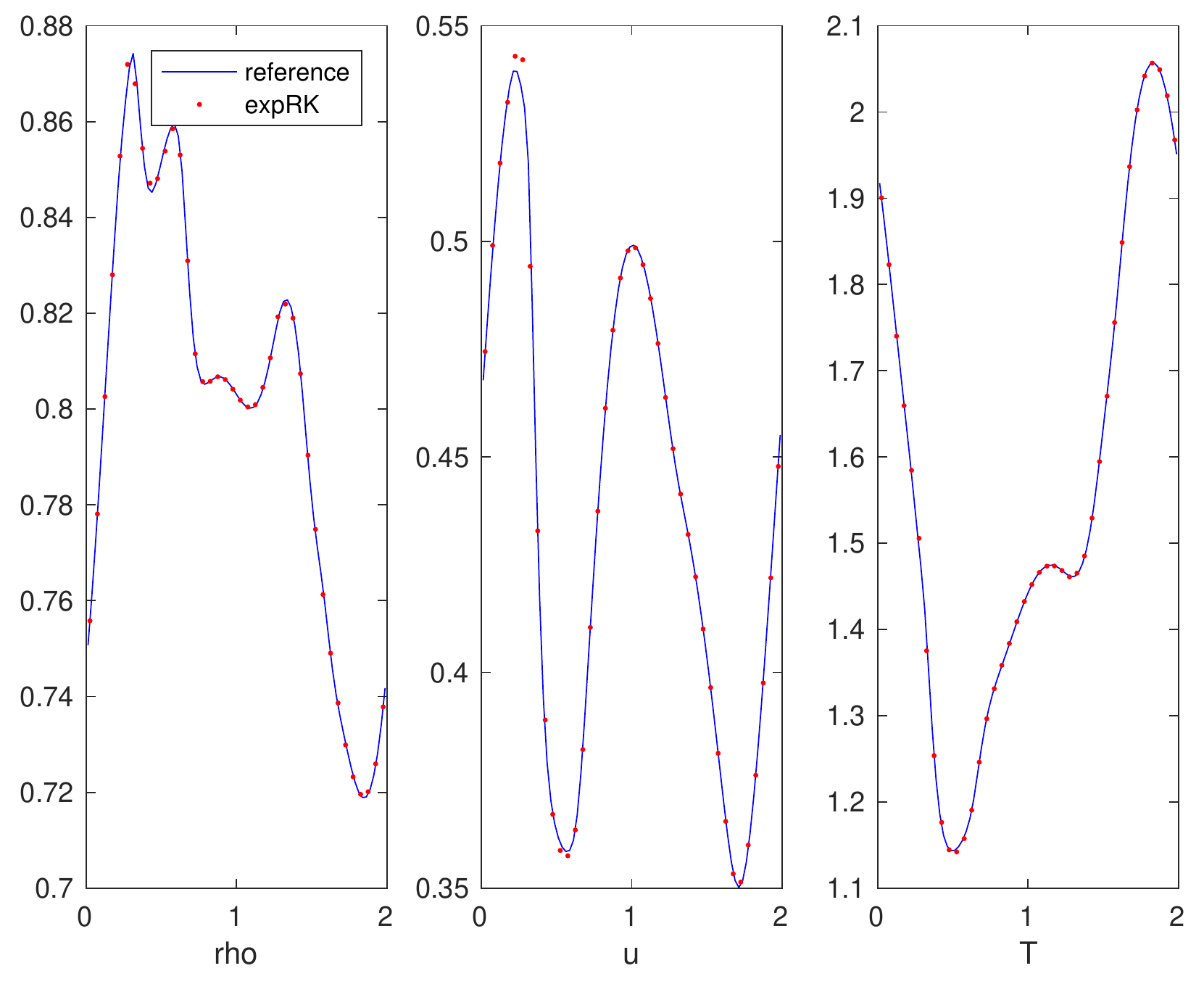}
	\caption{The BGK equation in a mixed regime. Left to right: density $\rho$, velocity $u$, and temperature $T$. Solid line: reference solution computed by the explicit SSP-RK2 scheme. Dots: solution computed by the proposed scheme.}
	\label{fig:APBGK}
\end{center}	
\end{figure}

For the Fokker-Planck equation, we consider the following $\varepsilon(x)$:
\begin{equation}
\varepsilon(x) = \varepsilon_0 + (\tanh(1-11(x-1))+\tanh(1+11(x-1))), \quad \varepsilon_0 = 5\times 10^{-4}.
\end{equation}
%Here we take a large $\varepsilon_0$ than the previous case (the BGK equation) because the parabolic CFL condition for the stiff FP operator is $\Delta t = O(\varepsilon\Delta v^2)$, which is more restrictive than the stiff BGK operator.
The numerical parameters are chosen the same as the BGK case, except in the reference solution, $\Delta t = \frac{1}{540}\frac{\Delta x}{|v|_{\text{max}}}\approx 3\times 10^{-6}$ in order to satisfy the explicit parabolic CFL condition. The result is shown in Figure 3, and again with good agreement.
\begin{figure}
\begin{center}
	\includegraphics[width=5.6in,height=2.5in]{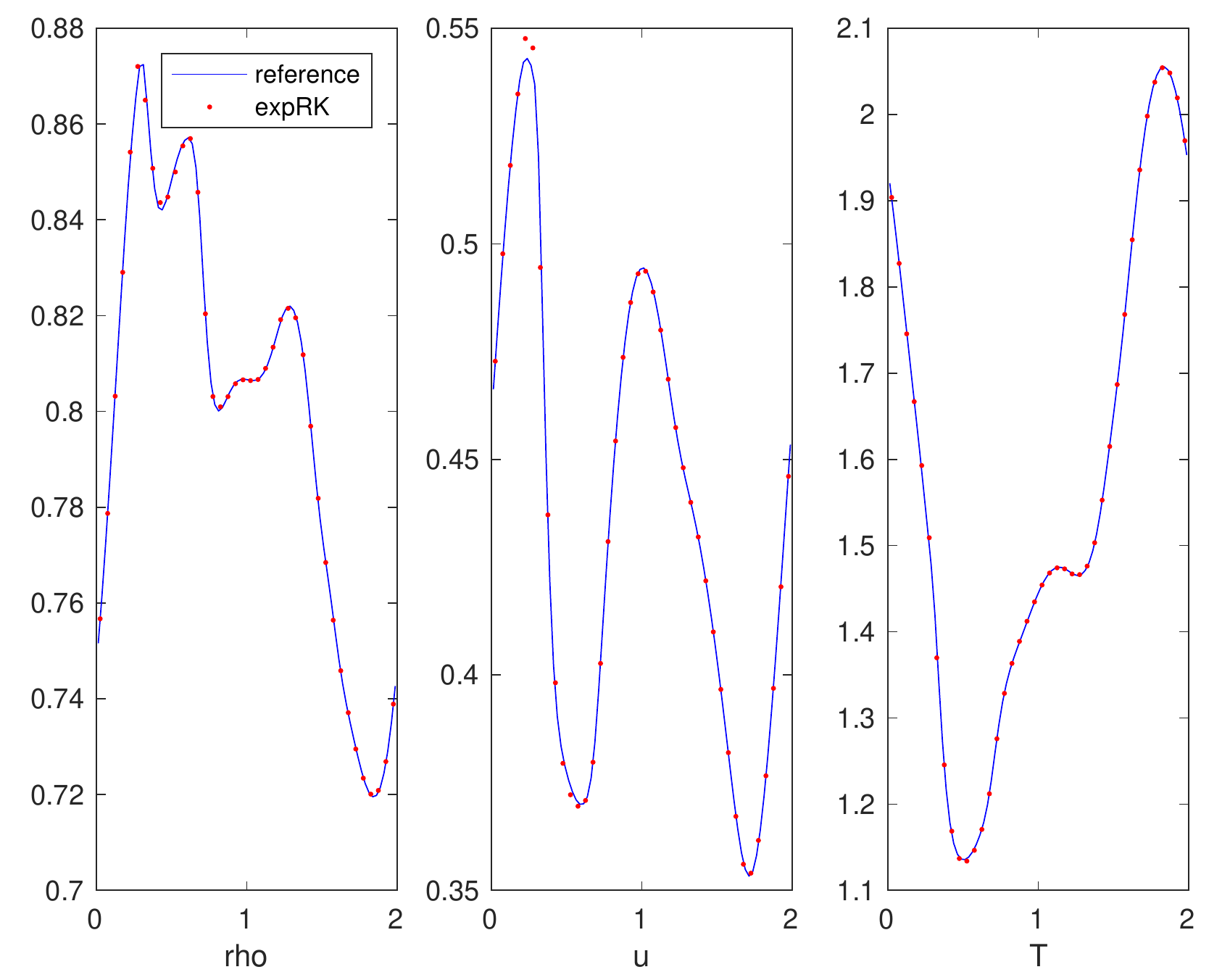}
	\caption{The Fokker-Planck equation in a mixed regime. Left to right: density $\rho$, velocity $u$, and temperature $T$. Solid line: reference solution computed by the explicit SSP-RK2 scheme. Dots: solution computed by the proposed scheme.}
%	\label{fig:negcell}
\end{center}	
\end{figure}

%%%%%%%%%%%%%%%%%%%%%%%%%%%%%%%%%%%%%%%%%%%%%%%%

\section{Conclusion}
\label{sec:con}

We introduced a new exponential Runge-Kutta time discretization method for a class of stiff kinetic equations. The method is second order, AP, and positivity-preserving. We applied the method to the relaxation type equations (BGK and ES-BGK equations), the diffusion type equations (kinetic Fokker-Planck and Vlasov-Poisson-Fokker-Planck equations), and even the full Boltzmann equation. Further, we showed that the method satisfies an entropy-decay property when coupled with upwind discretization for the transport term. Numerical examples for the BGK and Fokker-Planck equations were presented to demonstrate the properties of the proposed method.

\section*{Acknowledgement}

The first author would like to thank Prof.~Lili Ju for helpful discussion on exponential methods.

\bibliographystyle{plain}
\bibliography{hu_bibtex}
\end{document}